\documentclass[a4paper,11pt]{amsart}
\usepackage[colorlinks,linkcolor=blue,citecolor=blue]{hyperref}
\usepackage{latexsym, amssymb, amsmath, amsthm, mathrsfs, bbm}
\usepackage{amsmath, amscd}
\usepackage[all, knot]{xy}
\usepackage{dsfont}
\xyoption{arc}

\usepackage{color}

\usepackage{anysize}\marginsize{30mm}{30mm}{30mm}{30mm}
\def \To{\longrightarrow}
\def \Hom{\operatorname{Hom}}
\def \Vec{\operatorname{Vec}}

\def \id{\operatorname{id}}

\def \Z{\mathbb{Z}}
\def \unit{\mathds{1}}

\def \o{\omega}
\def \Arf{\operatorname{Arf}}

\numberwithin{equation}{section}

\newtheorem*{theorem*}{Theorem}
\newtheorem{theorem}{Theorem}[section]
\newtheorem{lemma}[theorem]{Lemma}

\newtheorem{remark}[theorem]{Remark}

\newcommand{\cx}{\mathbb{C}}
\newcommand{\cat}[1]{\mathcal{#1}}
\def \FPdim{\operatorname{FPdim}}
\newcommand{\FSexp}{\operatorname{FSexp}}
\newcommand{\Ob}{\operatorname{Ob}}
\newcommand{\DC}[1]{Z(\cat{#1})}
\newcommand{\Rep}{\operatorname{Rep}}

\begin{document}

\title{CLASSIFICATION OF SPHERICAL FUSION CATEGORIES OF FROBENIUS-SCHUR EXPONENT 2}

\author{Zheyan Wan}

\address{School of Mathematical Sciences, USTC, Hefei 230026, China} \email{wanzy@mail.ustc.edu.cn}

\author{Yilong Wang}

\address{Department of Mathematics,
Louisiana State University,
Baton Rouge, LA 70803, USA}
\email{yilongwang@lsu.edu}

\date{}

\begin{abstract}
In this paper, we propose a new approach towards the classification of spherical fusion categories by their Frobenius-Schur exponents. We classify spherical fusion categories of Frobenius-Schur exponent 2 up to monoidal equivalence. We also classify modular categories of Frobenius-Schur exponent 2 up to braided monoidal equivalence. It turns out that the Gauss sum is a complete invariant for modular categories of Frobenius-Schur exponent 2. This result can be viewed as a categorical analog of Arf's theorem on the classification of non-degenerate quadratic forms over fields of characteristic 2.
\end{abstract}
\maketitle

\section{Introduction}
Let $ \cat{C} $ be a spherical fusion category over $ \cx $. The higher Frobenius-Schur indicators $ \nu_n(V)$ of $ V \in \Ob(\cat{C}) $ and $n \in \Z $ are generalizations of the classical Frobenius-Schur indicator for irreducible finite group representations (see \cite{NS} and the references therein). The Frobenius-Schur indicators are important invariants of a spherical fusion category, especially when the category is in addition non-degenerately braided (in other words, modular). For example, the congruence subgroup conjecture on the $\operatorname{SL}(2, \Z)$ representations arising from modular categories can be resolved using generalized Frobenius-Schur indicators \cite{NS2010}.

The Frobenius-Schur exponent of a spherical fusion category $\cat{C}$, denoted by $\FSexp(\cat{C})$, is the smallest positive integer $n$ such that $\nu_n(V) = \dim_{\cat{C}}(V)$ for any object $V \in \Ob(\cat{C})$, where $ \dim_{\cat{C}}(V) $ is the categorical dimension of $ V $ in $ \cat{C} $. It is shown in \cite{NS} that $\FSexp(\cat{C})$ is equal to the order of the T-matrix of $\DC{C}$, the Drinfeld center of $\cat{C}$. Moreover, the Cauchy theorem for spherical fusion categories asserts that the prime ideals dividing $\FSexp(\cat{C})$ and those dividing the global dimension $\dim(\cat{C})$ are the same in the ring of algebraic integers \cite{BNRW-Rank-finiteness}. It is then reasonable to pursue a classification of spherical fusion categories by their Frobenius-Schur exponents, as opposed to the usual method of classification by rank \cite{Ostrik-Rank-2, Rowell-Stong-Wang, BNRW-Classification-by-rank}.

In this paper, we give a full classification of spherical fusion categories of Frobenius-Schur exponent 2. We show that such a spherical fusion category $\cat{C}$ is equivalent, as a fusion category, to $\Rep(\Z_2^n)$ for some positive integer $n$. In particular, the associativity constraints of $\cat{C}$ are all identities. We then show that if $\cat{C}$ is in addition modular, then $\cat{C}$ can be decomposed into a Deligne tensor product of two types of modular categories called $\cat{C}(\Z_2^2, q_1)$ and $\cat{C}(\Z_2^2, q_2)$. It is worth mentioning that in \cite[Theorem 3.2]{Bruillard-Rowell-Egyptian-fractions}, the authors showed that any modular category of Frobenius-Schur exponent 2 is a braided fusion subcategory of $ \Rep(D^\o (\Z_2^{2n})) $ for some positive integer $ n $. In this paper, we completely classify these modular categories by a categorical analog of Arf's theorem on the classification of non-degenerate quadratic forms over fields of characteristic 2. It turns out, in this case, the positive Gauss sum is a complete invariant.

The paper is structured as follows. In Section 2, we give a quick review of basic concepts and set up notations for future use.
We also discuss the braided monoidal structure on the category of $G$-graded vector spaces for a finite abelian group $G$. In Section 3, we classify spherical fusion categories of Frobenius-Schur exponent 2. Finally, in Section 4, we classify modular categories of Frobenius-Schur exponent 2.

\section{Preliminaries}
\subsection{Basic concepts and notations} \hfill

Now let $ \cat{C} $ be a \emph{fusion category} over $ \cat{C} $ over $ \cx $ \cite{EGNO}. In particular, $ \cat{C} $ is rigid monoidal, $ \cx$-linear, semisimple with finitely many isomorphism classes of simple objects such that the tensor unit $ \unit\in\Ob(\cat{C}) $ is simple. We fix a choice of representatives from the isomorphism class of simple objects and denote the set of all such representatives by $ \Pi_{\cat{C}} $. The \emph{Frobenius-Perron dimension} of $V\in\Ob(\cat{C})$, denoted by $\text{FP}\dim_{\cat{C}}(V)$, is the largest non-negative eigenvalue of the fusion matrix of $ V $. We define the Frobenius-Perron dimension of $ \cat{C} $ by $\text{FP}\dim(\cat{C}):=\sum_{V\in\Pi_{\cat{C}}}\text{FP}\dim_{\cat{C}}(V)^2$.

A fusion category $ \cat{C} $ is called \emph{spherical} if it has a pivotal structure such that the left and right pivotal traces coincide on all endomorphisms. In this case, the left (or right) pivotal trace of $ \id_V $, the identity of $ V\in\Ob(\cat{C}) $, is called the \emph{categorical dimension} of $ V $. We denote the categorical dimension of $ V\in\Ob(\cat{C}) $ by $ \dim_{\cat{C}}(V) $, and we define the \emph{global dimension} $ \cat{C} $ by $\dim(\cat{C}):=\sum_{V\in\Pi_{\cat{C}}}\dim_{\cat{C}}(V)^2$.

A spherical fusion category admitting a braiding is called a braided spherical fusion category (or premodular category). A braided spherical fusion category is called \emph{modular} if the braiding is non-degenerate, or equivalently, if its S-matrix is non-degenerate \cite{Muger}. For example, $\DC{C}$, the Drinfeld center of a spherical fusion category $\cat{C}$, is modular \cite{Muger}. Objects of $\DC{C}$ are pairs $(X,\sigma_X)$, where $X\in\Ob(\cat{C})$ and $\sigma_X:X\otimes-\stackrel{\sim}{\to}-\otimes X$ is a half braiding. Since the pivotal structure of $ \DC{C} $ is inherited from $ \cat{C} $, we have
\begin{equation}\label{eq:DimComparison}
	\dim_{\DC{C}}(V,\sigma_V)=\dim_{\cat{C}}(V)
\end{equation}
for any $ V\in\Ob(\cat{C}) $.

Let $ \cat{C} $ be a spherical fusion category. For any $ n \in \Z $, and for any $ V\in\Ob(\cat{C}) $, the \emph{$ n $-th Frobenius-Schur indicator} $\nu_n$ of $ V $ is defined to be the operator trace of a linear operator $E_V^{(n)}:\Hom_{\cat{C}}(\unit, V^{\otimes n})\to \Hom_{\cat{C}}(\unit, V^{\otimes n})$ satisfying $(E_V^{(n)})^n=\text{id}$. Here, $ V^{\otimes n} $ is understood as inductively defined by $ V^{\otimes (m+1)} = V \otimes V^{\otimes m} $ for $ 1 \leq m < n $, and associativity constraints are included in the definition of $ E_V^{(n)} $, see \cite{NS-Pivotal}. In particular, if $V$ is simple, then 
\begin{equation}\label{eq:nu1}
 \nu_1(V) = \delta_{\unit, V}.
\end{equation}
We also have 
\begin{equation}\label{eq:nu2}
\nu_2(V)=0, \mbox{ if } V\not\cong V^*,\ 
\nu_2(V)=1 \mbox{ or } -1, \mbox{ if } V\cong V^*
\end{equation}
for all $ V \in \Pi_{\cat{C}} $.

The \emph{Frobenius-Schur exponent of an object} $ V $ in a spherical fusion category $ \cat{C} $, denoted by $ \FSexp(V) $, is defined to be the smallest positive integer $n$ such that $\nu_n(V) = \dim_{\cat{C}}(V)$. The \emph{Frobenius-Schur exponent of $ \cat{C} $}, denoted by $ \FSexp(\cat{C}) $, is defined to be the smallest positive integer $n$ such that $\nu_n(V) = \dim_{\cat{C}}(V)$ for all $V\in\cat{C}$ \cite{NS}. When $ \cat{C} $ is the category of finite dimensional $H$-modules for a semisimple Hopf algebra $ H $ over $ \cx $, $ \FSexp(V) $ is equal to the exponent of $V$ as a finite dimensional $ H $-module. In other words, $ \FSexp(V) $ is equal to the exponent of the image of $ G $ in $ \operatorname{GL}(V, \cx) $ \cite{KSZ}.

It is immediate from the definition and Equation (\ref{eq:nu1}) that if $ \cat{C} $ is a spherical fusion category such that $ \FSexp(\cat{C}) = 1 $, then for any $ V \in \Pi_{\cat{C}} $, $ \dim_{\cat{C}}(V) = \delta_{\unit, V}$. According to \cite[Theorem 2.3]{ENO}, $ \dim_{\cat{C}}(V) \neq 0 $ for all $ V \in \Pi_{\cat{C}} $, hence $ \cat{C} $ has the tensor unit $ \unit $ as its only simple object. Therefore, $ \cat{C} $ is monoidally equivalent to $ \Vec_{\cx} $, the category of finite dimensional vector spaces over $ \cx $.

It is worth mentioning that by \cite{NS}, for any $ V \in\Ob(\cat{C}) $, $ \FSexp(V) $ does not depend on the choice of pivotal structures. In addition, $ \FSexp(\cat{C}) $ of a spherical fusion category $ \cat{C} $ depends only on the equivalence class of the modular category $ \DC{C} $.

\subsection{\texorpdfstring{Braided monoidal structure on $ G $-graded vector spaces}{}}\label{sec:gwc}\hfill

Let $G$ be a finite multiplicative abelian group. Recall that the category $\Vec_G^{\o}$ of finite-dimensional $G$-graded vector spaces has simple objects $\{V_g|g \in G\}$ where $(V_g)_h=\delta_{g,h}\mathbb{C}, \ \forall h \in G.$ The tensor product is given by $V_g \otimes V_h=V_{gh},$ and the tensor unit is $V_1$, where $1$ is the identity of $G$. The associator is given by a normalized 3-cocycle $ \o\in Z^3(G, \cx^{\times}) $
$$
\o(x,y,z):
V_x\otimes (V_y\otimes V_z)
\To
(V_x\otimes V_y )\otimes V_z.
$$

Now we equip $\Vec_G^{\o}$ with a braiding given by a normalized 2-cochain $ c \in C^2(G, \cx^{\times}) $
$$
c(x,y):
V_x\otimes V_y
\To
V_y\otimes V_x
$$
satisfying the hexagon axioms
\begin{equation}\label{EM3cocycle}
\frac{c(xy,z)}{c(x,z)c(y,z)}\frac{\o(x,y,z)\o(z,x,y)}{\o(x,z,y)}
=
1
=
\frac{c(x,yz)}{c(x,y)c(x,z)}\frac{\o(y,x,z)}{\o(x,y,z)\o(y,z,x)}
\end{equation}
for all $x,y,z\in G$. In other words, the pair $(\o,c)$ is an \emph{Eilenberg-MacLane 3-cocycle} of $G$. Finally, we equip $ \Vec_G^{\o} $ with the canonical (spherical) pivotal structure, which is simply given by identities on objects, so that the categorical dimensions are all positive. We denote this braided spherical fusion category by $ \Vec_G^{(\o, c)} $.

An Eilenberg-MacLane 3-cocycle $(\o,c)$ is called a \emph{coboundary} if there exists a 2-cochain $h \in C^2(G, \cx^{\times})$ such that
\begin{equation}\label{EM3coboundary}
\o=\delta h
\mbox{ and }
c(x,y)=\frac{h(x,y)}{h(y,x)}.
\end{equation}
The \emph{Eilenberg-MacLane cohomology group} $H^3_{ab}(G,\mathbb{C}^{\times})$ is then defined by
$$
H^3_{ab}(G,\mathbb{C}^{\times})
=
Z^3_{ab}(G,\mathbb{C}^{\times})/B^3_{ab}(G,\mathbb{C}^{\times}),
$$
where $Z^3_{ab}(G,\mathbb{C}^{\times})$ and $B^3_{ab}(G,\mathbb{C}^{\times})$ are respectively the abelian groups of Eilenberg-MacLane 3-cocycles and 3-coboundaries.
To $(\o,c)\in Z^3_{ab}(G,\mathbb{C}^{\times})$, one can assign the function $q(x):=c(x,x)$, called its \emph{trace}. It is easy to show that $ q(x) $ is a \emph{quadratic form} (or a quadratic function). In other words, we have
\begin{enumerate}
\item
$q(x^a)=q(x)^{a^2}$ for any $a\in\mathbb{Z}$, and
\item
$b_q(x,y):=\frac{q(xy)}{q(x)q(y)}$ defines a bicharacter of $G$.
\end{enumerate}
We will use the pair $ (G, q) $ to denote a quadratic form $ q $ on the finite abelian group $ G $. When the group $ G $ is clear from the context, we will sometimes simply write $ q $. Note that given two quadratic forms $ (G, q) $ and $ (G', q') $, we can define a quadratic form on $ G \oplus G' $, denoted by $ q \oplus q' $, via the following formula:
$$
(q\oplus q')(x, x') := q(x)q'(x')
$$
for all $ (x, x') \in G \oplus G' $. The quadratic form $ (G \oplus G', q \oplus q') $ is called the \emph{direct sum} of $ (G, q) $ and $ (G', q') $.

We recall a theorem of Eilenberg-MacLane (\cite{EMa} and \cite{EMb}).

\begin{theorem*}[Eilenberg-MacLane]
The map assigning $(\o,c)$ to its trace induces an isomorphism of groups
$$
H^3_{ab}(G,\mathbb{C}^{\times})
\stackrel{\cong}{\To}Q(G,\mathbb{C}^{\times})
$$
where $Q(G,\mathbb{C}^{\times})$ is the abelian group of quadratic forms from $G$ to $\mathbb{C}^{\times}$.
\end{theorem*}

We introduce the following notations before proceeding. Given a group homomorphism $ f: G \to G' $ and a positive integer $ n $, we use the standard notation for the $ n $-fold product of $ f $:
$$
f^n: G^n \to (G')^n,\  f^n(g_1, ..., g_n) := (f(g_1), ..., f(g_n)).
$$
For any $ n $-cochain $ \mu\in C^n(G', \cx^\times) $, we define $ f^\ast(\mu) := \mu \circ f^n $.

Two quadratic forms $q:G\to\mathbb{C}^{\times}$ and $q':G'\to\mathbb{C}^{\times}$ are \emph{equivalent} if there exists a group isomorphism $f:G\to G'$ such that $q =f^\ast(q')$.

\begin{lemma}\label{equivalent}
$\Vec_G^{(\o,c)}$ and $\Vec_{G'}^{(\o',c')}$ are equivalent braided monoidal categories if and only if the traces of $(\o,c)$ and $(\o',c')$ are equivalent quadratic forms.
\end{lemma}
\begin{proof}
If $F:\Vec_G^{(\o,c)}\to\Vec_{G'}^{(\o',c')}$ is a braided monoidal equivalence with the natural isomorphism $\mu(x,y):F(V_x)\otimes F(V_y)\to F(V_x\otimes V_y)$, then $F$ induces a group isomorphism $f:G\to G'$ on simple objects.
Moreover, the following diagrams commute:
$$
\begin{CD}
(F(V_x)\otimes F(V_y))\otimes F(V_z)
	@>\mu(x,y)\otimes\id >>
F(V_x\otimes V_y)\otimes F(V_z)
	@>\mu(xy,z)>>
F((V_x\otimes V_y)\otimes V_z)
	\\
	@A\o'(f(x),f(y),f(z))AA    @.    @AA F(\o(x,y,z))A\\
F(V_x)\otimes (F(V_y)\otimes F(V_z))
	@> \id \otimes \mu(y,z)>>
F(V_x)\otimes F(V_y\otimes V_z)
	@>\mu(x,yz)>>
F(V_x\otimes(V_y\otimes V_z))
\end{CD}
$$

$$
\begin{CD}
F(V_x)\otimes F(V_y)
@>c'(f(x),f(y))>>
F(V_y)\otimes F(V_x)\\
@V\mu(x,y)VV @VV\mu(y,x)V\\
F(V_x\otimes V_y)
@> F(c(x,y))>>
F(V_y\otimes V_x)
\end{CD}
$$
Hence $f^*(\o')=\o \cdot \delta\mu$ and $f^*(c')(x,y)=c(x,y)\frac{\mu(x,y)}{\mu(y,x)}$. Therefore, $(\o,c)$ and $(f^*(\o'),f^*(c'))$ differ by an Eilenberg-MacLane 3-coboundary. By the theorem of Eilenberg-MacLane, $q=f^*(q')$.

Conversely, assume there exists a group isomorphism $f:G\to G'$ such that $q=f^*(q')$. By the theorem of Eilenberg-MacLane, $(\o,c)$ and $(f^*(\o'),f^*(c'))$ differ by an Eilenberg-MacLane 3-coboundary. In other words, there exists a 2-cochain $\mu$ of $G$ such that $f^*(\o')=\o \cdot \delta\mu$ and $f^*(c')(x,y)=c(x,y)\frac{\mu(x,y)}{\mu(y,x)}$.
Define $F(V_x):=V_{f(x)}$ and $\mu(x,y):F(V_x)\otimes F(V_y)\to F(V_x\otimes V_y)$, then $F$ together with $ \mu $ extends to a braided monoidal equivalence between $\Vec_G^{(\o,c)}$ and $\Vec_{G'}^{(\o',c')}$.
\end{proof}

\begin{remark}
In the light of the Eilenberg-MacLane Theorem, we will denote any representative in the braided monoidal equivalence class of some $\Vec_G^{(\o,c)}$ by $\cat{C}(G,q)$ where $q$ is the trace of $(\o,c)$. Then Lemma \ref{equivalent} can be rewritten as follows: $ \cat{C}(G, q) \cong \cat{C}(G', q')$ as braided monoidal categories if and only if $ q $ and $ q' $ are equivalent quadratic forms.
\end{remark}

\section{Classification of spherical fusion categories of Frobenius-Schur exponent 2}
In this section, we classify spherical fusion categories of Frobenius-Schur exponent 2 up to monoidal equivalence. Let $\mathcal{C}$ be such a category. The Frobenius-Schur exponent of $\DC{C}$ is then also 2 by Corollary 7.8 of \cite{NS}. Consequently, for any $ V \in \Ob(\cat{C}) $, $\nu_2(V)=\dim_{\cat{C}}(V)$. In addition, if $V$ is simple, then  $\nu_2(V)=0,\pm1$ (cf. Equation (\ref{eq:nu2})). By \cite[Theorem 2.3]{ENO}, $\dim_{\cat{C}}(V)\ne0$. Hence, we have
\begin{equation}\label{eq:pm1}
	\dim_{\cat{C}}(V)=\nu_2(V)=\pm1
\end{equation}
for any $ V \in \Pi_{\cat{C}} $. By Proposition 8.22 of \cite{ENO},
$$
(\text{FP}\dim(\mathcal{C}))^2=\frac{(\dim(\mathcal{C}))^2}{\dim_{\DC{C}}((V, \sigma_V))^2}
$$
for some $(V, \sigma_V)\in\Pi_{\DC{C}}$. Since $ (V, \sigma_V)\in\Pi_{\DC{C}}$ implies that $ V\in\Pi_{\cat{C}}$ \cite{Muger}, by Equations (\ref{eq:DimComparison}) and (\ref{eq:pm1}), we have $(\text{FP}\dim(\mathcal{C}))^2=(\dim(\mathcal{C}))^2$. As both $\text{FP}\dim(\mathcal{C})$ and $\dim(\mathcal{C})$ are positive \cite[Theorem 2.3]{ENO}, we have $\text{FP}\dim(\mathcal{C})=\dim(\mathcal{C})$. Hence, $\mathcal{C}$ is pseudo-unitary \cite{ENO}. By Proposition 8.23 of \cite{ENO}, there exists a unique spherical pivotal structure on $\mathcal{C}$ such that $\dim_{\cat{C}}(V)=\text{FP}\dim_{\cat{C}}(V)>0$ for all $V\in \Pi_{\mathcal{C}}$. Since our classification is up to monoidal equivalence, we can assume without loss of generality that $ \cat{C} $ is equipped with its unique spherical pivotal structure described above.

According to Equation (\ref{eq:pm1}), for any $ V\in\Pi_{\cat{C}} $, $ V $ is self-dual. As a result, we have
$$
\dim_{\cat{C}}(V \otimes V^{\ast}) = \dim_{\cat{C}}(V \otimes V) = \dim_{\cat{C}}(V)^2 = 1.
$$
By rigidity, pseudo-unitarity and the fact that categorical dimension is a character of the fusion ring, we have $ V \otimes V \cong \unit $. Therefore, $\Pi_{\mathcal{C}}$ is a group of exponent 2, or $\Pi_{\mathcal{C}}=\mathbb{Z}_2^n$ for some positive integer $ n $.
As a result, $\mathcal{C}=\Vec_{\mathbb{Z}_2^n}^{\o}$ for some $ \o\in H^3(\Z_2^n, \cx^{\times}) $. By Theorem 9.2 of \cite{NS}, for any finite group $ G $, we have
$$
\text{FS}\exp(\Vec_G^{\o})
=
\text{lcm}_{g\in G}\text{ord}(\o|_{\langle g\rangle})\text{ord}(g),
$$
where $ \o|_{\langle g\rangle} $ denotes the restriction of $ \o $ on the subgroup generated by $ g $. Since $\text{FS}\exp(\mathcal{C})=2$, we have $\o|_{\langle x\rangle}$ is trivial for all $x\in\mathbb{Z}_2^n$.

For any $ n \in \Z $, consider the map
\begin{align*}
b: H^3(\Z_2^n, \cx^\times)
&\longrightarrow
\{\pm 1\}^{2^{n}-1}
\\
\lambda
&\mapsto
(..., \lambda_C, ...)
\end{align*}
where $ C $ ranges over the subgroups of $ \Z_2^n $ of order 2, and
$$
\lambda_C =
\begin{cases}
1 &\ \mbox{if the restriction of } \lambda \mbox{ on } C \mbox{ is trivial,}\\
-1 &\ \mbox{otherwise.}
\end{cases}
$$

By \cite[Proposition 2.2]{Mason}, $ b $ is injective. Therefore, $ \o|_{\langle x \rangle} $ being trivial for all $ x \in \Z_2^n $ implies that $\o$ itself is cohomologous to the trivial 3-cocycle. Let $ [\o] $ be the cohomology class of $ \o $ in $ H^3(G, \cx^\times) $, we have $ [\o] = 1 $, and $ \Vec_{\mathbb{Z}_2^n}^{\o} $ is monoidally equivalent to $ \Vec_{\mathbb{Z}_2^n}^{1} $ by a standard argument. Note that the more familiar category of finite dimensional representations of $ \Z_2^n $, denoted by $ \Rep(\mathbb{Z}_2^n) $, is nothing but an incarnation of $\Vec_{\mathbb{Z}_2^n}^{1}$ as a fusion category.

We summarize the above discussion in the following theorem.

\begin{theorem}\label{thm:spherical}
If $\mathcal{C}$ is a spherical fusion category of Frobenius-Schur exponent 2, then $\mathcal{C}$ is pseudo-unitary. Moreover, $ \cat{C} $ is monoidally equivalent to $\Rep(\mathbb{Z}_2^n)$ for some positive integer $n$.
\qed
\end{theorem}

\begin{remark}
We can also obtain this result by the explicit formula of the normalized 3-cocycle \cite{HWY},
$$\o(x,y,z)=\prod_{r=1}^n(-1)^{a_ri_r[\frac{j_r+k_r}{2}]}\prod_{1\le r<s\le n}(-1)^{a_{rs}k_r[\frac{i_s+j_s}{2}]}\prod_{1\le r<s<t\le n}(-1)^{a_{rst}k_rj_si_t}$$
where $x=(i_1,\dots,i_n)$, $y=(j_1,\dots,j_n)$, $z=(k_1,\dots,k_n)$, $i_r,j_r,k_r,a_r,a_{rs},a_{rst}\in\{0,1\}$.

$$\o(x,x,x)=\prod_{r=1}^n(-1)^{a_ri_r^2}\prod_{1\le r<s\le n}(-1)^{a_{rs}i_ri_s}\prod_{1\le r<s<t\le n}(-1)^{a_{rst}i_ri_si_t}=1.$$
Take $x=(0,\dots,0,1,0,\dots,0)$ where 1 is at the $r$-th position, we get $a_r=0$ for $1\le r\le n$.
Take $x=(0,\dots,0,1,0,\dots,0,1,0,\dots,0)$ where the first 1 is at the $r$-th position, the second 1 is at the $s$-th position, we get $a_{rs}=0$ for $1\le r<s\le n$.
Take $x=(0,\dots,0,1,0,\dots,0,1,0,\dots,0,1,0,\dots,0)$ where the first 1 is at the $r$-th position, the second 1 is at the $s$-th position, the third 1 is at the $t$-th position, we get $a_{rst}=0$ for $1\le r<s<t\le n$. Hence $[\o] = 1$.
\end{remark}

\section{Classification of modular categories of Frobenius-Schur exponent 2}
In this section, we use the result in the previous section to classify modular categories of Frobenius-Schur exponent 2 up to braided monoidal equivalence. Let $\mathcal{C}$ be such a modular category. By the same argument as in the previous section, $ \cat{C} $ is pseudo-unitary, and we will equip $ \cat{C} $ with its canonical spherical pivotal structure such that $ \dim_{\cat{C}}(V) = \FPdim_{\cat{C}}(V)>0 $ for all $ V \in \Pi_{\cat{C}} $. According to Theorem \ref{thm:spherical}, $ \cat{C} $ is equivalent to $ \Rep(\Z_2^n) $ as a fusion category for some $ n $. Consequently, as a braided fusion category, $ \cat{C} \cong \Vec_{\mathbb{Z}_2^n}^{(\o, c)} $ for some Eilenberg-MacLane 3-cocycle $ (\o, c) $. By the same argument as in the previous section, $ [\o] = 1 $.

Therefore, $ \cat{C} \cong \Vec_{\Z_2^n}^{(1, c)}  $ with $ (1, c) $ an Eilenberg-MacLane 3-cocycle. By Equation \eqref{EM3cocycle}, we have $c(1,x)=c(x,1)=1$, and $q(x)^2=c(x,x)^2=1$ for all $x\in\mathbb{Z}_2^n$, in particular, $ q$ takes value in $\{\pm 1\} $. Therefore, by definition (cf. Section \ref{sec:gwc}), the bilinear form associated to $ q $ is given by 
$$
b_q: \Z_2^n \oplus \Z_2^n \to \{\pm 1 \},
\ 
b_q(x, y) = \frac{q(xy)}{q(x)q(y)}=c(x,y)c(y,x)
$$
for any $ (x, y)\in \Z_2^n \oplus\Z_2^n $. Moreover, since $b_q(x, y)$ is the entry of the S-matrix of $\mathcal{C}$ \cite{EGNO}, the modularity of $ \cat{C} $ then implies that $q$ is a non-degenerate quadratic form. Hence, $ b_q $ is a non-degenerate alternating form (in particular, $ b_q(x,x) = 1 $ for any $ x \in \Z_2^n $). Therefore, $ n=2m $ is even. Moreover, there exists a symplectic basis $ \{e_1, ..., e_m, f_1, ..., f_m\} $ of $ \Z_2^{2m} $, with respect to which $ b_q(e_j, e_k) = b_q(f_j, f_k) = 1 $, and $ b_q(e_j, f_k) = (-1)^{\delta_{j,k}} $ for any $ j, k = 1, ..., m $.

For any non-degenerate quadratic form $ q: \Z_2^{2m} \to \{\pm 1 \} $, we define its additive version $ Q: \Z_2^{2m} \to \Z_2$ such that $ (-1)^{Q(x)} = q(x) $ for any $ x \in \Z_2^{2m} $. Then the \emph{Arf invariant} of $q$, denoted by $ \Arf(q) $, is given by the classical Arf invariant of $ Q $. More precisely, we have
$$
\Arf(q)
:=
\Arf(Q)
=
\sum_{j = 1}^{m} Q(e_j)Q(f_j),
$$
where $ \{ e_1, ..., e_m, f_1, ..., f_m \} $ is the symplectic basis given above. Note that the Arf invariant takes value in $ \Z_2 $, where we use the standard notation $ \Z_2 = \{ 0, 1 \} $. We also view $ \Z_2 $ as a field here.

Arf showed in \cite{Arf} that the Arf invariant is independent of the choice of basis, and is additive with respect to the direct sum of quadratic forms. More importantly, Arf showed that the dimension $ 2m $ (of $ \Z_2^{2m} $ as a vector space over $ \Z_2 $) and the Arf invariant $ \Arf(q) $ completely determine the equivalence class of a non-degenerate quadratic form $ (\Z_2^{2m}, q) $ over $ \Z_2 $. The readers, especially those who are not fluent in German, are highly recommended to consult Appendix 1 of \cite{MilnorHusemoller} for a beautiful exposition of Arf invariant.

As a consequence of Arf's theorems, for any positive integer $ m $, there are only two equivalence classes of non-degenerate quadratic forms on $ \Z_2^{2m} $, and they can be obtained as direct sums from two inequivalent quadratic forms on $\Z_2^2$. We give explicit representatives for the two equivalence classes of non-degenerate quadratic forms on $ \Z_2^2 $ as follows:

\begin{equation}\label{eq:q1}
q_1: \Z_2^2\to\{\pm 1\},\  q_1(x,y)= (-1)^{xy}
\end{equation}
and

\begin{equation}\label{eq:q2}
q_2: \Z_2^2\to \{\pm 1\},\ q_2(x,y)= (-1)^{x^2+xy+y^2}
\end{equation}
for any $ x, y \in\Z_2 $. In other words, we have $ Q_1(x, y) = xy $ and $ Q_2(x, y) = x^2+xy+y^2 $. Therefore, any quadratic form $ (\Z_2^{2m}, q) $ is equivalent to $ q_1^a \oplus q_2^{m-a}$ for some $ a \geq 0 $. The presentation of $ q $ may not be unique, but they are all equivalent to the representatives given as follows.

By direct computation, we have $ \Arf(q_1) = 0 $, $ \Arf(q_2) = 1 $. Therefore, $ \Arf(q_1\oplus q_1) = \Arf(q_2 \oplus q_2) = 0 $. Since both $ q_1\oplus q_1 $ and $ q_2\oplus q_2 $ are quadratic forms on $ \Z_2^4 $, by Arf's theorem, $ q_1\oplus q_1 $ is equivalent to $ q_2\oplus q_2 $. As a result, if a non-degenerate quadratic form $ (\Z_2^{2m}, q) $ is equivalent to $ q_1^{a} \oplus q_2^{m-a}$ for some $ a \geq 0 $, then its Arf invariant is given by
$$
\Arf(q) =
\begin{cases}
0, \mbox{ if } m-a \mbox{ is even},\\
1, \mbox{ otherwise}.
\end{cases}
$$
by the additivity of the Arf invariant. Now that we can change any summand of the form $ q_2 \oplus q_2 $ into $ q_1 \oplus q_1 $ without changing the equivalence class of $ q $, we have $ q $ is equivalent to $ q_1^m$ if $ \Arf(q) = 0 $, and $ q $ is equivalent to $ q_1^{m-1} \oplus q_2 $ if $ \Arf(q) = 1 $. We will assume for the rest of this article, that any non-degenerate quadratic form $ (\Z_2^{2m}, q) $ is represented in this way.

Next, we analyze the categorical interpretation of the direct sum of quadratic forms (c.f. Section \ref{sec:gwc}). Let $ (G, q) $ and $ (G', q') $ be two non-degenerate quadratic forms. We consider the Deligne tensor product of the modular categories $ \cat{C}(G, q) $ and $ \cat{C}(G', q') $, denoted by $\cat{D} := \cat{C}(G, q) \boxtimes \cat{C}(G', q')$ \cite{EGNO}. By definition, $ \cat{D} $ is also a modular category, and its fusion rule is given by the multiplication of the abelian group $ G \oplus G' $. Therefore, $ \Pi_{\cat{D}} = G \oplus G'$, hence $ \cat{D} \cong \Vec_{G \oplus G'}^{(\o, c)} $ for some Eilenberg-MacLane 3-cocycle $ (\o, c) $ of $ G \oplus G' $. Let $ p(x) = c(x, x) $ be the corresponding trace. In other words, $ \cat{D} \cong \cat{C}(G \oplus G', p) $.

Let $ (\o_1, c_1) $ and $ (\o_2, c_2) $ be representatives of the Eilenberg-MacLane 3-cohomology classes corresponding to $ q $ and $ q' $ respectively. By the definition of the Deligne tensor product, the associativity constraints in $ \cat{D} $ is the tensor product of those in $ \cat{C}(G, q) $ and $ \cat{C}(G', q') $. In other words, for any $ (x_1, x_2), (y_1, y_2), (z_1, z_2) \in G \oplus G' $, we have
$$
\o((x_1, x_2), (y_1, y_2), (z_1, z_2))
=
\o_1(x_1,y_1,z_1)\o_2(x_2,y_2,z_2).
$$
Similarly, we have the following equality from the definition of the braiding on $ \cat{D} $
$$
c((x_1,x_2),(y_1,y_2))=c_1(x_1,y_1)c_2(x_2,y_2).
$$
In particular, for any $ (x_1, x_2)\in G \oplus G' $, we have
$$
p(x_1, x_2) = q(x_1)q'(x_2) = (q\oplus q')(x_1, x_2).
$$
Therefore, by Lemma \ref{equivalent}, we have $ \cat{D} \cong \cat{C}(G \oplus G', p) \cong \cat{C}(G\oplus G', q \oplus q') $.

We summarize the above discussion in the following lemma.

\begin{lemma}
$\mathcal{C}(G \oplus G', q \oplus q') \cong \mathcal{C}(G, q) \boxtimes \mathcal{C}(G', q')$ as modular categories.
\qed
\end{lemma}

Combining the discussions in this section gives rise to the following classification result.

\begin{theorem}\label{thm:modular}
If $\mathcal{C}$ is a modular category of Frobenius-Schur exponent 2, then $\mathcal{C}$ is pseudo-unitary, and $ \cat{C} $ is braided monoidally equivalent to $\mathcal{C}(\Z_2^{2m},q)$ for a positive integer $ m $ and a non-degenerate quadratic form $ q $. Moreover, we have the following Deligne tensor product decomposition
$$
\cat{C} \cong
\begin{cases}
\cat{C}(\Z_2^2, q_1)^{\boxtimes m}
&\mbox{if } \Arf(q) = 0,\\
&\\
\cat{C}(\Z_2^2, q_1)^{\boxtimes (m-1)} \boxtimes \cat{C}(\Z_2^2, q_2)
&\mbox{if } \Arf(q) = 1,\\
\end{cases}
$$
where $ q_1 $ and $q_2$ are given in Equations (\ref{eq:q1}) and (\ref{eq:q2}).
\qed
\end{theorem}

\begin{remark}
A braiding of $\cat{C}(\Z_2^2, q_1)$ can be given by
$$
c_1((x,y),(a,b))=(-1)^{xb},
$$
and a braiding of $\cat{C}(\Z_2^2, q_2)$ can be given by
$$
c_2((x,y),(a,b))=(-1)^{xa+yb+ay}.
$$
\end{remark}

We would like to interpret the Arf invariant in the modular category setting. Firstly, note that for any non-degenerate quadratic form $ (\Z_2^{2m}, q) $, by direct computation, we have
$$
(-1)^{\Arf(q)}
=
\frac{1}{\sqrt{|\Z_2^{2m}|}}\sum_{x\in \Z_{2}^{2m}} q(x)
=
\frac{1}{2^{m}}\sum_{x\in \Z_{2}^{2m}} q(x)
$$
(by Arf's theorems, we only have to check this equality for $ (\Z_2^2, q_1) $ and$ (\Z_2^2, q_2) $, which is immediate). In the literature, the above quantity is also referred to as the \emph{Gaussian sum} for the quadratic form $ q $ on the finite abelian group $ \Z_2^{2m} $ (for example, see \cite{Scharlau}).

On the category-theoretical side, recall (for example, \cite{EGNO}) that the \emph{positive Gauss sum} of a modular category $\cat{C}$ is defined by
$$
\tau_+
:=
\sum_{X\in \Pi_{\cat{C}}}\theta_X\dim_{\cat{C}}(X)^2,
$$
where $ \theta_X $ is the twist of the simple object $ X $. It is standard \cite{BK} that in a modular category $ \cat{C} $, the global dimension is the square of the complex absolute value of $ \tau_+ $. In other words,
$$
\dim(\cat{C}) = |\tau_+|^2.
$$
The \emph{multiplicative central charge} of $\cat{C}$ is defined by
$$
\xi
:=
\frac{\tau_+(\cat{C})}{\sqrt{\dim(\cat{C})}}
=
\frac{\tau_+}{|\tau_+|}.
$$
Note that $ \xi(\cat{C}) $ is well-defined as $ \dim(\cat{C})$ is a totally positive algebraic integer \cite{EGNO}.

In particular, when $ \cat{C} = \cat{C}(\Z_2^{2m}, q) $ for a non-degenerate quadratic form $ (\Z_2^{2m}, q)$, we can compute the dimension $ m $ and the Arf invariant $ \Arf(q) $ of $ (\Z_2^{2m}, q) $ by the positive Gauss sum $ \tau_+ $ as follows. We have $ \Pi_{\cat{C}} = \Z_2^{2m} $. We also have, for any $ x \in \Z_2^{2m} $, that $ \dim_{\cat{C}}(x) = 1 $, hence
\begin{equation}\label{eq:dim}
|\tau_+|^2
=
\dim(\cat{C})
=
\sum_{x\in \Z_{2}^{2m}}\dim_{\cat{C}}(x)^2
=
|\Z_{2}^{2m}|
=
2^{2m},
\end{equation}
in particular, $ |\tau_+| = 2^{m} $, or $ m = \log_2(|\tau_+|) $. Moreover, since for any $ x \in \Z_2^{2m} $, $ \theta_x = q(x) $ \cite{EGNO}, we have
\begin{equation}\label{eq:charge}
\frac{\tau_+}{2^{m}}
=
\frac{\tau_+}{|\tau_+|}
=
\xi(\cat{C}(\Z^{2m}, q))
=
\frac{1}{\sqrt{|\Z_{2}^{2m}|}}\sum_{x\in \Z_{2}^{2m}} q(x)
=
(-1)^{\Arf(q)}.
\end{equation}
Hence, $ \Arf(q) $ is 0 or 1 depending on whether $ \tau_+ $ is positive or negative, respectively.

Conversely, by Equations (\ref{eq:dim}) and (\ref{eq:charge}), we have $ \tau_+ = (-1)^{\Arf(q)}2^m $.

The argument above shows that both the dimension and the Arf invariant of the quadratic form $ (\Z_2^{2m}, q) $ are completely determined by positive Gauss sum $ \tau_+ $ of the modular category $ \cat{C}(\Z_2^{2m}, q) $ and vice versa.

Recall that by Arf, a non-degenerate quadratic form is completely determined (up to equivalence) by its dimension and its Arf invariant. In the same vein, we restate Theorem \ref{thm:modular} as a categorical analog of Arf's theorem.

\begin{theorem}\label{thm:ArfInCategory}
	If $\mathcal{C}$ is a modular category of Frobenius-Schur exponent 2, then $\mathcal{C}$ is pseudo-unitary, and $ \cat{C} $ is completely determined, up to braided monoidal equivalence, by its positive Gauss sum $ \tau_+ $. More precisely, we have
	
	$$
	\cat{C} \cong
	\begin{cases}
	\cat{C}(\Z_2^2, q_1)^{\boxtimes \log_2(|\tau_+|)}
	&\mbox{if } \tau_+ > 0,\\
	&\\
	\cat{C}(\Z_2^2, q_1)^{\boxtimes (\log_2(|\tau_+|)-1)} \boxtimes \cat{C}(\Z_2^2, q_2)
	&\mbox{if } \tau_+ < 0.\\
	\end{cases}
	$$
	\qed
\end{theorem}

Finally, we make a remark on the prime factorization of modular categories.

A modular category is \emph{non-trivial} if its rank is larger than 1. A non-trivial modular category called a \emph{prime modular category} if it is not braided monoidally equivalent to a Deligne tensor product of two non-trivial modular categories. 

A direct consequence of Theorem \ref{thm:modular} is that there are only two (pseudo-unitary) prime modular categories of Frobenius-Schur exponent 2. In view of \cite[Lemma 2.4]{BNRW-Rank-finiteness}, there are finitely many prime modular categories of Frobenius-Schur exponent 2.

\section{Acknowledgments}
We thank Professor Siu-Hung Ng for helpful conversations. ZW gratefully acknowledges the support from NSFC grants 11431010, 11571329.

\bibliographystyle{abbrv}
\bibliography{FSexp2}

\end{document}